\documentclass[12pt,twoside]{amsart}
\usepackage{amsfonts}
\usepackage{amsmath, amsthm, amssymb, mathrsfs, paralist,tabularx,supertabular, verbatim, amsthm, manfnt}
\usepackage{url}
\usepackage{comment}
\usepackage[all]{xy}

\newcommand{\Z}{{\mathbb Z}}

\newcommand{\F}{{\mathbb F}}

\include{python}

\evensidemargin \oddsidemargin
\parskip=\medskipamount

\newtheorem{thm}{Theorem}[section]
\newtheorem{theorem}[thm]{Theorem}

\newtheorem{lemma}[thm]{Lemma}	
\newtheorem{proposition}[thm]{Proposition}

\theoremstyle{definition}
\newtheorem{definition}[thm]{Definition}

\theoremstyle{remark}

\newtheorem*{lemma*}{Lemma}

\numberwithin{equation}{section}

\title{Variations on a question concerning the degrees of divisors of $x^n-1$}

\author{Lola Thompson}
\address{Department of Mathematics\\ 
6188 Kemeny Hall\\
Dartmouth College\\
Hanover, NH 03755, USA}
\email[] {Lola.Thompson@Dartmouth.edu}

\begin{document}

\begin{abstract} 
\noindent In this paper, we examine a natural question concerning the divisors of the polynomial $x^n-1$: ``How often does $x^n-1$ have a divisor of every degree between $1$ and $n$?'' In a previous paper, we considered the situation when $x^n-1$ is factored in $\Z[x]$. In this paper, we replace $\Z[x]$ with $\F_p[x]$, where $p$ is an arbitrary-but-fixed prime. We also consider those $n$ where this condition holds for all $p$.\end{abstract}

\maketitle 

\section{Introduction and statement of results}

Which polynomials have divisors of every degree in a given polynomial ring? In a previous paper \cite{thompson}, we answered this question in $\Z[x]$ for the family of polynomials $f(x) = x^n-1$, where $n$ ranges over all positive integers. We defined an integer $n$ to be \textit{$\varphi$-practical} if the polynomial $x^n-1$ has a divisor in $\Z[x]$ of every degree up to $n$ and we showed that, if $F(X) = \#\{n \leq X: n \ \hbox{is} \ \varphi\hbox{-practical}\}$, then there exist two positive constants $c_1$ and $c_2$ such that \begin{equation}\label{lt} c_1 \frac{X}{\log X} \leq F(X) \leq c_2 \frac{X}{\log X}.\end{equation} 

In this paper, we will examine the factorization of $x^n-1$ over other rings. For each rational prime $p$, we will define an integer $n$ to be \textit{$p$-practical} if $x^n-1$ has a divisor in $\F_p[x]$ of every degree less than or equal to $n$. In order to better understand the relationship between $\varphi$-practical and $p$-practical numbers, we will define an intermediate set of numbers which we shall call the $\lambda$-practical numbers. An integer $n$ is \textit{$\lambda$-practical} if and only if it is $p$-practical for every rational prime $p$. Clearly each $\varphi$-practical number is $\lambda$-practical. In Sections 2 and 3, we will give alternative characterizations of the $p$-practical and $\lambda$-practical numbers that are often easier to work with.

The main goal of this paper is to examine the relative sizes of the sets of $\varphi$-practical, $\lambda$-practical and $p$-practical numbers. Accordingly, the remaining sections take the following form. In Section 4, we will develop some theory on the structure of $\lambda$-practical numbers and show that there are infinitely many $\lambda$-practical numbers that are not $\varphi$-practical. We will go one step further in Section 5 and prove:

\begin{theorem}\label{philambdatheorem} For $X$ sufficiently large, the order of magnitude of $\lambda$-practicals in $[1, X]$ that are not $\varphi$-practical is $\frac{X}{\log X}$.\end{theorem}

In Section 6, we examine the relationship between $p$-practicals and $\lambda$-practicals, culminating in a proof of the following theorem:

\begin{theorem}\label{plambdatheorem} For every rational prime $p$ and for $X$ sufficiently large, the order of magnitude of $p$-practicals in $[1, X]$ that are not $\lambda$-practical is at least $\frac{X}{\log X}$.\end{theorem}

We remark that the order of magnitude described in Theorem \ref{plambdatheorem} is deemed to be \textit{at least} $\frac{X}{\log X}$ (as opposed to \textit{precisely} $\frac{X}{\log X}$) because of the fact that we still do not know the true order of magnitude of the $p$-practical numbers. In another paper \cite{thompson2}, we show (assuming the validity of the Generalized Riemann Hypothesis) that $$\frac{X}{\log X} \ll \#\{n \leq X: n \ \hbox{is} \ p\hbox{-practical}\} \ll X \sqrt{\frac{\log \log X}{\log X}}.$$ As a result, Theorem \ref{plambdatheorem} tells us that the order of magnitude of $p$-practicals that are not $\lambda$-practical is between $\frac{X}{\log X}$ and $X \sqrt{\frac{\log \log X}{\log X}}$ (with the upper bound holding provided that the Generalized Riemann Hypothesis is valid). 

Throughout this paper, we will make use of the following notation. Let $n$ be a positive integer. Let $p$ and $q$, as well as any subscripted variations, be primes. We will use $P(n)$ to denote the largest prime factor of $n$, with $P(1) = 1.$ Moreover, we will use $P^-(n)$ to denote the smallest prime factor of $n$, with $P^-(1) = + \infty$. Let $\tau(n)$ denote the number of positive divisors of $n$, and let $\Omega(n)$ denote the number of prime factors of $n$ which are not necessarily distinct. 

\section{Background and preliminary results}

In \cite{thompson}, we gave the following alternative characterization for the $\varphi$-practical numbers, which we state here as a lemma.

\begin{lemma}\label{phicond} An integer $n$ is \textit{$\varphi$-practical} if and only if every $m$ with $1 \leq m \leq n$ can be written in the form $$m = \sum_{d \in \mathcal{D}} \varphi(d),$$ where $\mathcal{D}$ is a subset of divisors of $n$. \end{lemma} It is not difficult to see Lemma \ref{phicond}: since $$x^n-1 = \prod_{d \mid n} \Phi_d(x),$$ where $\Phi_d(x)$ is the $d^{th}$ cyclotomic polynomial (which is irreducible in $\Z[x]$ with degree $\varphi(d)$), we see that divisors of $x^n-1$ correspond to subsets of divisors of $n$. 

The term ``$\varphi$-practical'' was chosen in recognition of the connection between this alternative characterization of $\varphi$-practical numbers and the definition of a practical number. A. K. Srinivasan coined the term ``practical number'' in 1948, defining an integer $n$ to be \textit{practical} if every $m$ with $1 \leq m \leq \sigma(n)$ can be written as a sum of distinct positive divisors of $n$; that is, $m = \sum_{d \in \mathcal{D}} d$, where $\mathcal{D}$ is a subset of divisors of $n$. Six years later, B. M. Stewart \cite{stewart} gave a simple necessary-and-sufficient condition for integers to be practical:

\begin{lemma}[Stewart]\label{stewartkey}If $M$ is a practical number and $p$ is a prime with $(p, M) = 1$, then $M' = p^k M$ is practical (for $k \geq 1$) if and only if $p \leq \sigma(M) + 1$. Moreover, if $M$ is practical, so is $M/P(M)$.\end{lemma}

Stewart's condition gave rise to a number of results concerning the practical numbers, most notably a series of improvements on upper and lower bounds for the size of the set of practical numbers up to $X$. The tightest bounds were given by E. Saias in \cite{saias}, who showed that there exist two constants $C_1$ and $C_2$ such that \begin{equation}\label{saias} C_1 \frac{X}{\log X} \leq PR(X) \leq C_2 \frac{X}{\log X},\end{equation} where $PR(X) = \#\{n \leq X : n \ \mathrm{is} \ \mathrm{practical}\}$. 

Recall the upper and lower bounds for $F(X)$ given in \eqref{lt}, which mirror Saias' bounds for the practical numbers. One of our aims in this paper will be to obtain a similar upper bound for the $\lambda$-practical numbers. As in the case of the $\varphi$-practical numbers, we will find it helpful to have alternative characterizations of the $\lambda$-practical and $p$-practical numbers in terms of their divisors. Let $\ell_a(n)$ denote the multiplicative order of $a$ (mod $n$) for integers $a$ with $(a, n) = 1$. If $(a, n) > 1$, let $n_{(a)}$ denote the largest divisor of $n$ that is coprime to $a$, and let $\ell_a^*(n) = \ell_a(n_{(a)}).$ In particular, if $(a, n) = 1$ then $\ell_a^*(n) = \ell_a(n).$ \begin{lemma}\label{pcondsum} An integer $n$ is \textit{$p$-practical} if and only if every $m$ with $1 \leq m \leq n$ can be written as $m = \sum_{d \mid n} \ell_p^*(d) n_d$, where $n_d$ is an integer with $0 \leq n_d \leq \frac{\varphi(d)}{\ell_p^*(d)}.$\end{lemma} To see the relationship between the two characterizations of $p$-practical numbers, recall the following well-known proposition (cf. \cite[pg. 489, ex.20]{df}):

\begin{proposition}\label{factorppractical} The following two cases completely characterize the factorization of $\Phi_d(x)$ over $\F_p$:

1) If $(d, p) =1$, then $\Phi_d(x)$ decomposes into a product of distinct irreducible polynomials of degree $\ell_p(d)$ in $\F_p[x]$.

2) If $d = mp^k$, $(m, p) = 1$, then $\Phi_d(x) = \Phi_m(x)^{\varphi(p^k)}$ over $\F_p$.

\end{proposition} Thus, the correspondence between the definitions follows from the fact that each cyclotomic polynomial $\Phi_d(x)$ dividing $x^n-1$ factors into $\varphi(d)/\ell_p^*(d)$ irreducible polynomials of degree $\ell_p^*(d)$ over $\F_p[x].$  

As we will discuss in the next section, the $\lambda$-practical numbers can be defined in a similar manner. However, it takes a bit more work to prove this. 

\section{An alternative characterization for the $\lambda$-practical numbers}

Just as we showed that the $\varphi$-practical and $p$-practical numbers have alternative characterizations that resemble the definition of a practical number, we can also show that the $\lambda$-practical numbers have such a characterization. Let $\lambda(n)$ denote the universal exponent of the multiplicative group of integers modulo $n.$ We will show that the following theorem gives a criterion for an integer $n$ to be $\lambda$-practical that is equivalent to the definition that we gave in Section 1: \begin{theorem}\label{lambda} An integer $n$ is \textit{$\lambda$-practical} if and only if we can write every integer $m$ with $1 \leq m \leq n$ in the form $m = \sum_{d \mid n} \lambda(d) m_{d}$, where $m_{d}$ is an integer with $0 \leq m_{d} \leq \frac{\varphi(d)}{\lambda(d)}$.\end{theorem} Before presenting the proof, however, we will pause to ponder a related question. We can think of the set of integers $n$ that are ``$p$-practical for all primes $p$'' as the intersection between all of the sets of integers that are $p$-practical. In addition to describing the intersection of these sets, we can also describe their union. 

\begin{proposition}\label{union} For each prime $p$, let $S_p$ be the set of $p$-practical numbers. Then $$\bigcup_{p \ \mathrm{prime}} S_p = \Z_+.$$ \end{proposition}


In fact, we can prove a stronger result: it turns out that each integer $n$ is $p$-practical for infinitely many values of $p$. Namely, for a given $n$, Dirichlet's Theorem on Primes in Arithmetic Progressions \cite[pp. 119]{pollack} implies that there are infinitely many primes $p$ for which $p \equiv 1 \pmod{n}.$ In other words, $\ell_p(n) = 1$ for infinitely many primes $p$, so $x^n-1$ splits completely into linear factors in $\F_p[x]$ for infinitely many primes $p$. This argument implies that each integer $n$ is $p$-practical for a positive proportion of the $p$'s. We could also observe that, if $p \equiv -1 \pmod{n}$, then $\ell_p(n) = 2$, hence all of the irreducible factors of $x^n-1$ have degree at most $2$. Dirichlet's Theorem also guarantees the existence of infinitely many such primes. We remark that there are integers $n$ for which the set of primes $p \equiv \pm 1 \pmod{n}$ are the only primes for which $x^n-1$ has a divisor of every degree ($n = 5$ is the smallest such integer). 

We will now show that $\bigcup_{p \ \mathrm{prime}} S_p$ is precisely the set of integers satisfying the conditions given in Theorem \ref{lambda}. 

\begin{lemma}\label{key} For all positive integers $n$, there exists a prime $p$ such that $\ell_p^*(d) = \lambda(d)$ for all $d \mid n$. \end{lemma}

\begin{proof} First, we will consider the case where $n = q^{e}$, where $q$ is an odd prime. Each divisor of $n$ is of the form $d = q^{f}$, with $0 \leq f \leq e.$ Since $(\Z/q^{f} \Z)^\times$ is cyclic, there must be some element $a \in (\Z/q^{e} \Z)^\times$ such that $\ell_a^*(q^{e}) = \lambda(q^{e}).$ But $a$ is also a generator for $(\Z/q^{f} \Z)^\times$, i.e. $\ell_a^*(q^{f}) = \lambda(q^{f}).$ By Dirichlet's Theorem, there exists a prime $p \equiv a \pmod{q^{e}}.$ Thus, we can certainly find a prime $p$ with $\ell_p^*(q^{f}) = \lambda(q^{f})$ for all $f$ with $0 \leq f \leq e.$ 

If $n = 2^{e}$, we observe that, when $p = 3$, we have $\ell_p^*(2^j) = \lambda(2^j)$ for all $j \geq 1$. Hence, $\ell_3^*(d) = \lambda(d)$ for all divisors $d$ of $2^{e}$. 

Now we consider the case where $n = q_1^{e_1} \cdots q_k^{e_k}$, $k \geq 2$. Each $d \mid n$ can be written in the form $d = q_1^{f_1} \cdots q_k^{f_k}$, where $0 \leq f_i \leq e_i$ holds for $i = 1,...,k.$ For each $i$, if $q_i$ is odd, let $a_i$ be a primitive root (mod $q_i^{e_i}$), and if $q_i = 2$, take $a_i = 3$. Since $q_1,...,q_k$ are pairwise relatively prime then, by the Chinese Remainder Theorem, there exists an integer $x$ with \begin{align*} x &\equiv a_1 \pmod{q_1^{e_1}} \\
&\hspace{5.5 pt} \vdots \\
x & \equiv a_k \pmod{q_k^{e_k}}.\end{align*} By Dirichlet's Theorem, there exists a prime $p$ with $p \equiv x \pmod{n}$. In other words, $\ell_p^*(q_i^{e_i}) = \lambda(q_i^{e_i})$ for $i = 1,...,k.$ As remarked earlier, we have $\ell_p^*(q_i^{f_i}) = \lambda(q_i^{f_i})$ for all $f_i$ with $0 \leq f_i \leq e_i.$ Therefore, since $q_1,...,q_k$ are pairwise relatively prime, we have $$\ell_p^*(q_1^{f_1} \cdots q_k^{f_k}) = \mathrm{lcm} [\ell_p^*(q_1^{f_1}),\cdots,\ell_p^*(q_k^{f_k})] = \mathrm{lcm}[\lambda(q_1^{f_1}),\cdots,\lambda(q_k^{f_k})] = \lambda(q_1^{f_1} \cdots q_k^{f_k}).$$\end{proof} 

Below, we provide the proof of Theorem \ref{lambda}.

\begin{proof} If $n$ is $\lambda$-practical then, by Lemma \ref{key}, there exists a prime $p'$ such that $\ell_{p'}^*(d) = \lambda(d)$ for all $d \mid n$. Since $n$ is $p$-practical for all primes $p$ then, in particular, $n$ is $p'$-practical, i.e. for all integers $m$ with $1 \leq m \leq n$, we have $$m = \sum_{d \mid n} \ell_{p'}^*(d) n_{p'}(d),$$ where $n_{p'}(d)$ is an integer satisfying $0 \leq n_{p'}(d) \leq \frac{\varphi(d)}{\ell_{p'}^*(d)}.$ Thus, for all $m$ with $1 \leq m \leq n$, we have $$m = \sum_{d \mid n} \lambda(d) n_{p'}(d),$$ since $\ell_{p'}^*(d) = \lambda(d)$ for all $d \mid n$. Since it is necessarily the case that $0 \leq n_{p'}(d) \leq \frac{\varphi(d)}{\ell_{p'}^*(d)} = \frac{\varphi(d)}{\lambda(d)}$, then $n$ satisfies the condition given in Theorem \ref{lambda}. 

On the other hand, suppose that every integer $m$ with $1 \leq m \leq n$ can be written in the form $m = \sum_{d \mid n} \lambda(d) m_d,$ where $m_d$ is an integer satisfying $0 \leq m_d \leq \frac{\varphi{d}}{\lambda(d)}.$ By definition, $\lambda(d) = \mathrm{max}_{a \in (\Z/d\Z)^\times} \ell_a^*(d).$ Since $\ell_a^*(d) \leq \lambda(d)$ for all $a$ in $(\Z/d\Z)^\times$, then certainly every $m$ with $1 \leq m \leq n$ can be written in the form $m = \sum_{d \mid n} \ell_p^*(d) n_d$, where $p$ is any rational prime and $0 \leq n_d \leq \frac{\varphi(d)}{\ell_p^*(d)}.$ Thus, $n$ is $\lambda$-practical.  \end{proof}

\section{Key lemmas}

In this section, we provide some key lemmas for characterizing the relationship between $\varphi$-practical and $\lambda$-practical numbers. These lemmas will be used in section \ref{lambdaphisection} in order to obtain information about the relative asymptotic densities of these sets. We begin by reminding the reader of some useful results on the $\varphi$-practical numbers. In \cite{thompson}, we proved the following necessary condition for an integer $n$ to be $\varphi$-practical: 

\begin{lemma}\label{necessary} Suppose that $n = p_1^{e_1} \cdots p_k^{e_k}$ is $\varphi$-practical, where $p_1 < p_2 < \cdots < p_k$ and $e_i \geq 1$ for $i = 1,...,k$. Define $m_i = p_1^{e_1} \cdots p_i^{e_i}$ for $i = 0,...,k-1$. Then, the inequality $p_{i+1} \leq m_i + 2$ must hold for all $i$. \end{lemma} 

We say that an integer $n$ is \textit{weakly $\varphi$-practical} if all of its prime factors satisfy the inequality from Lemma \ref{necessary}. We note that the inequality in Lemma \ref{necessary} also gives a necessary condition for a positive integer $n$ to be $\lambda$-practical. Namely, if $p_{i+1} > m_i + 2$ for some $i$ such that $0 \leq i \leq k-1$ then, since $\lambda(p_{i+1}) = \varphi(p_{i+1}) = p_{i+1} - 1$, we have $\lambda(p_{i+1}) > m_i +1$. Since $m_i = \sum_{d \mid m_i} \lambda(d) \frac{\varphi(d)}{\lambda(d)}$, then $m_i + 1$ cannot be written as a sum of $\lambda(d)$'s, so such an $n$ would not be $\lambda$-practical. Thus, we have proven the following:

\begin{lemma}\label{lnecessary} Every $\lambda$-practical number is weakly $\varphi$-practical. \end{lemma}

The converse to Lemma \ref{lnecessary} is false. For example, $n = 9$ is weakly $\varphi$-practical but not $\lambda$-practical. However, we can show that the converse holds for even integers and for squarefree integers. In order to complete these proofs, we will need a lemma on the structure of $\lambda$-practical numbers. The following result (cf. \cite{thompson}[Lemma 4.1]) gives a partial characterization for the structure of $\varphi$-practical numbers:

\begin{lemma}\label{keylemmaphi} If $M$ is $\varphi$-practical and $p$ is prime with $(p, M) = 1$, then $M' = pM$ is $\varphi$-practical if and only if $p \leq M + 2$. Moreover, $M' = p^{k}M, k \geq 2$ is $\varphi$-practical if and only if $p \leq M + 1$. \end{lemma}

The statement of Lemma \ref{keylemmaphi} mirrors the statement of Lemma \ref{stewartkey}, with one important difference: although Lemma \ref{stewartkey} yields a necessary-and-sufficient condition for all numbers to be practical, Lemma \ref{keylemmaphi} cannot be used to generate the full list of $\varphi$-practical numbers. For example, $3^2 \times 5 \times 17 \times 257 \times 65537 \times (2^{31} - 1)$ is $\varphi$-practical, but none of the numbers $3^2$, $3^2 \times 5$, $3^2 \times 5 \times 17$, $3^2 \times 5 \times 17 \times 257$, or $3^2 \times 5 \times 17 \times 257 \times 65537$ are $\varphi$-practical. Nevertheless, Lemma \ref{keylemmaphi} is useful in obtaining a lower bound for the number of $\varphi$-practical numbers up to $X$. Likewise, to find a lower bound for the number of $\lambda$-practical numbers that fail to be $\varphi$-practical (which we shall accomplish in Section \ref{lambdaphisection}), we will make use of the following lemma:

\begin{lemma}\label{weak lambda} Let $n = mp$, where $m$ is $\lambda$-practical, $p \leq m + 2$ and $(p, m) = 1$. Then $n$ is $\lambda$-practical. Moreover, if $n = p^{k}m$ with $k \geq 2$, then $n$ is $\lambda$-practical if $p \leq m + 1.$ \end{lemma}

The proof of Lemma \ref{weak lambda} is virtually identical to the proof of Lemma \ref{keylemmaphi}. The idea is to use the characterization of $\lambda$-practical numbers given in Theorem \ref{lambda} to show that every integer $l \in [1,n]$ can be expressed in the form $$l = \sum_{d \mid m} \lambda(d) m_d,$$ with $0 \leq m_d \leq \frac{\varphi(d)}{\lambda(d)}.$ In order to check that this holds when $n = mp$, we observe that if every $l \in [1,n]$ can be written in the form \begin{equation}\label{l1} l = (p-1) Q + R, \ 0 \leq Q, R \leq m\end{equation} then, using our hypothesis that $m$ is $\lambda$-practical, we have \begin{equation}\label{l2identity} l = \sum_{d \mid m} (p-1) \lambda(d) m_d + \sum_{d \mid m} \lambda(d) m'_d,\end{equation} where $0 \leq m_d, m'_d \leq \frac{\varphi(d)}{\lambda(d)}.$ We can use the facts that $\lambda(p) = p-1$ and $\lambda(p_1^{e_1} \cdots p_k^{e_k}) = \hbox{lcm}[\lambda(p_1^{e_1}),...,\lambda(p_k^{e_k})]$ to show that we can re-write \eqref{l2identity} in the following manner: $$l = \sum_{d \mid m} \lambda(pd) m_{pd} + \sum_{d \mid m} \lambda(d) m'_d,$$ where $0 \leq m_{pd} \leq \frac{\varphi(pd)}{\lambda(pd)}$ and $0 \leq m'_d \leq \frac{\varphi(d)}{\lambda(d)}.$ Thus, the proof boils down to showing that every $l \in [1,n]$ can be expressed as in \eqref{l1}, which follows from breaking $[1,n]$ into subintervals of the form $[(p-1)Q, (p-1)Q + m]$ and using the hypothesis that $p \leq m+2$ to show that the subintervals cover the full interval. The higher power case is similar, but requires induction on the power of the prime $p$. 

\begin{proposition}\label{even} Let $n$ be an even integer. Then $n$ is weakly $\varphi$-practical if and only if $n$ is $\varphi$-practical if and only if $n$ is $\lambda$-practical. \end{proposition}

\begin{proof} We will begin by showing that an even integer $n$ is weakly $\varphi$-practical if and only if it is $\varphi$-practical. If $n$ is even and weakly $\varphi$-practical, then we can write $n = p_1^{e_1} \cdots p_k^{e_k}$, where $2 = p_1 < p_2 < \cdots < p_k$ and $e_i \geq 1$ for $i = 1,...,k.$ We will use induction on the number of distinct prime factors of $n$ to show that $n$ is $\varphi$-practical. For our base case, we observe that $2^{e_1}$ is $\varphi$-practical for all positive values of $e_1.$ For our induction hypothesis, we assume that $m = p_1^{e_1} \cdots p_{k-1}^{e_{k-1}}$ is $\varphi$-practical. Since $m$ is even and $p_k$ is odd, then $p_k \leq m + 2$ implies that $p_k \leq m + 1$. Thus, by Lemma \ref{keylemmaphi}, $m p_k^{e_k}$ is $\varphi$-practical. The other direction follows immediately from Lemma \ref{necessary}. The proof for $\lambda$-practicals is the same, this time using Lemma \ref{weak lambda} instead of Lemma \ref{keylemmaphi}.\end{proof}

As we remarked above, the conditions given in Lemma \ref{necessary} for an integer $n$ to be weakly $\varphi$-practical are necessary, but not sufficient, for $n$ to be $\varphi$-practical. However, when $n$ is squarefree, we have shown (cf. Corollary 4.2 in \cite{thompson}) that these notions are equivalent. There is an analogous situation for $\lambda$-practical numbers. 

\begin{proposition} Let $n$ be a squarefree integer. Then $n$ is $\lambda$-practical if and only if it is $\varphi$-practical. \end{proposition}

\begin{proof} We showed in Corollary 4.2 of \cite{thompson} that a squarefree integer is $\varphi$-practical if and only if it is weakly $\varphi$-practical. From Lemma \ref{lnecessary}, every $\lambda$-practical number is weakly $\varphi$-practical. The other direction of the proof is trivial, as all $\varphi$-practical numbers are automatically $\lambda$-practical. \end{proof}

While it is easy to see that all $\varphi$-practical numbers are $\lambda$-practical, the converse does not hold. In fact, it is easy to show that there are infinitely many counterexamples:

\begin{proposition}\label{infl} There are infinitely many $\lambda$-practical numbers that are not $\varphi$-practical. \end{proposition}

\begin{proof} Let $X \geq 1$ be a real number. Let $n = 45 \cdot \prod_{23 < p \leq X} p.$ It follows from Bertrand's Postulate that every prime $p \mid n$ with $p \nmid 45$ satisfies $p \leq m + 2$, where $m$ is the product of $45$ and all of the primes $23 < q < p.$ Then, since $45$ is $\lambda$-practical, it follows from Lemma \ref{weak lambda} that $n$ is $\lambda$-practical. However, $n$ is not $\varphi$-practical, since $x^{45} - 1$ has no divisor of degree $22$ and all other primes $p \mid n$ are greater than $23$, so $\lambda(p) > 22$. Thus, as we let $X$ tend to infinity, we see that this method produces an infinite family of $\lambda$-practical numbers that are not $\varphi$-practical. \end{proof}

We will elaborate on the ideas presented in the proof of Proposition \ref{infl} in the next section when we determine the asymptotic density of $\lambda$-practical numbers that fail to be $\varphi$-practical. 

\section{Proof of Theorem \ref{philambdatheorem}}\label{lambdaphisection}

In this section, we will discuss the distribution of $\lambda$-practical numbers in relation to that of the $\varphi$-practical numbers. We begin by reminding the reader of the method of proof in \eqref{lt}, which will be a model for some of the arguments that we will use to bound the number of $\lambda$-practical integers up to $X$. The key to proving the upper bound in \eqref{lt} was to use Proposition \ref{even} in order to show that all even $\varphi$-practical numbers are practical. To handle the case of odd $\varphi$-practicals, we observed that, for every odd integer $n$ in $(0, X]$, there exists a unique positive integer $l$ such that $2^ln$ is in the interval $(X, 2X]$. Moreover, we showed that $2^ln$ is $\varphi$-practical if $n$ is $\varphi$-practical. As a result, we were able to construct a one-to-one map from the set of odd $\varphi$-practical numbers in $(1, X]$ to a subset of the even $\varphi$-practical numbers in $(X, 2X]$.This allowed us to directly compare the size of the set of $\varphi$-practical numbers with the size of the set of practical numbers, which we knew to be $O(X/\log X)$ from \cite[Theorem 2]{saias}.

We can use the same argument to show that the upper bound given in \eqref{lt} will also serve as an upper bound for the number of $\lambda$-practical numbers up to $X$. The only modification needed is to use Proposition \ref{even} to show that all even $\lambda$-practical numbers are practical. On the other hand, Lemma \ref{lnecessary} shows that, if $n$ is an odd $\lambda$-practical number, then it is weakly $\varphi$-practical. Thus, for $l \geq 1$, $2^ln$ is weakly $\varphi$-practical, since multiplying a weakly $\varphi$-practical integer $n$ by a power of $2$ will not prevent its prime divisors from satisfying the inequalities from Lemma \ref{necessary}. Therefore, we can use the argument given above for the odd $\varphi$-practicals to obtain the same upper bound for the number of $\lambda$-practicals up to $X$. 

In order to obtain a lower bound, we simply observe that the set of $\varphi$-practical numbers is properly contained within the set of $\lambda$-practical numbers. Hence, the lower bound that we gave in \cite{thompson} for the $\varphi$-practical numbers will also serve as a lower bound for the $\lambda$-practical numbers. As a result, we have:

\begin{proposition} Let $F_\lambda(X) = \# \{n \leq X : n$ is $\lambda$-practical$\}$. Then, there exist positive constants $c_3$ and $c_4$ such that $$c_3 \frac{X}{\log X} \leq F_\lambda(X) \leq c_4 \frac{X}{\log X},$$ for all $X \geq 2$.\end{proposition}

Note that the argument above shows that we may, in fact, take $c_3 = c_1$ and $c_4 = c_2$ (where $c_1$ and $c_2$ are the constants from \eqref{lt}). However, this does not imply that $F_\lambda(X) - F(X) = o(\frac{X}{\log X}).$ In fact, as stated in Theorem \ref{philambdatheorem}, we can show that $F_\lambda(X) - F(X) \gg \frac{X}{\log X}.$ Before we prove this result, we remind the reader of some definitions and lemmas used in the lower bound argument for the $\varphi$-practical numbers in \cite{thompson}, which will be useful in this scenario as well.

Let $1 = d_1(n) < d_2(n) < \cdots  < d_{\tau(n)}(n) = n$ denote the increasing sequence of divisors of a positive integer $n.$ We define $$T(n) = \mathrm{max}_{1 \leq i < \tau(n)} \frac{d_{i+1}(n)}{d_i(n)}.$$ \begin{definition} An integer $n$ is called $2$-\textit{dense} if $n$ is squarefree and $T(n) \leq 2.$\end{definition} Note that any $2$-dense number $n > 1$ is even. Let $$D(X) = \#\{1 \leq n \leq X : n \ \mathrm{is} \ 2\hbox{-dense} \}.$$ 
\noindent In \cite{saias}, Saias built on techniques developed by Tenenbaum (cf. \cite{tenenbaum0}, \cite{tenenbaum}) in order to prove the following upper and lower bounds for $D(X)$:

\begin{lemma}[Saias]\label{SaiasD} There exist positive constants $\kappa_1$ and $\kappa_2$ such that $$\kappa_1 \frac{X}{\log X} \leq D(X) \leq \kappa_2 \frac{X}{\log X}$$ for all $X \geq 2$. \end{lemma}

The lower bound for the $2$-dense integers is also a lower bound for the count of practical numbers up to $X$, since every $2$-dense integer automatically satisfies Stewart's condition. Unfortunately, we cannot extrapolate any information about a lower bound for $F(X)$ from the lower bound for $D(X)$ because the $2$-dense integers are not necessarily $\varphi$-practical (for example, $n = 66$). In \cite{thompson}, we got around this problem by adopting the following modification on the definition of $2$-dense:

\begin{definition}\label{std} A $2$-dense number $n$ is \textit{strictly $2$-dense} if $\frac{d_{i+1}}{d_i} < 2$ holds for all $i$ satisfying $1 < i < \tau(n) - 1.$\end{definition} 
\noindent We showed (cf. Lemma 5.4 in \cite{thompson}) that the strictly $2$-dense integers have an important relationship with the $\varphi$-practical numbers: 

\begin{lemma}\label{modified} Every strictly $2$-dense number is $\varphi$-practical. \end{lemma}

Recall that, in Proposition \ref{infl}, we constructed an infinite family of $\lambda$-practical numbers that are not $\varphi$-practical. We will use this construction, along with several lemmas concerning the $2$-dense and strictly $2$-dense numbers, in order to prove the first of our main theorems.

\begin{proof}[Proof of Theorem \ref{philambdatheorem}] We begin by recalling an argument given in \cite{thompson}. Let $n = mpj,$ where $m$ is a $2$-dense integer, $p$ is a prime satisfying $m < p < 2m$, and $j$ is an integer that has the following properties: $j \leq X/mp$, $P^-(j) > p$ and $mpj$ is $2$-dense. Let $C > 16$ be an integer that is chosen to be large relative to the size of the constant $\kappa_1$ from Lemma \ref{SaiasD}. For each integer $k > C$, we consider those $2$-dense numbers $m \in (2^{k-1}, 2^k]$. Since $m < p < 2m$, we must have $p \in (2^{k-1}, 2^{k+1}).$ We say that $n$ has an \textit{obstruction at $k$} if $m$ and $p$ land within these intervals, i.e., if $p$ is a prime in our construction that might prevent $n$ from being strictly $2$-dense. In Theorem 5.11 in \cite{thompson}, we showed that, if $C$ is large enough, the number of $2$-dense integers with obstructions at $k > C$ is negligible relative to the full count of $2$-dense integers. Thus, consider the set $$\mathcal{N} = \{n \leq X : n \ \hbox{is $2$-dense with no obstructions at} \ k > C\}.$$ For an appropriate choice of $C \geq 5$, we have $$\# \mathcal{N} \geq \kappa \frac{X}{\log X},$$ where $\kappa > 0$ is some absolute constant. As in \cite{thompson}, we define a function $f : \mathcal{N} \rightarrow \Z_+$ to be a function that maps each element $n \in \mathcal{N}$ to its largest $2$-dense divisor with all prime factors less than or equal to $2^C$. Let $\mathcal{M} = \mathrm{Im} f.$ The Pigeonhole Principle guarantees that there is some $m_0 \in \mathcal{M}$ that has at least $$\frac{\# \mathcal{N}}{\# \mathcal{M}} \geq \frac{\kappa}{4^{2^C}} \frac{X}{\log X}$$ elements in its preimage, since the Chebyshev bound (cf. \cite[pg. 108]{pollack}) implies that $\prod_{p \leq 2^C} p \leq 4^{2^C}.$ In other words, $m_0 = f(n)$ for at least the average number of integers in a pre-image. 

Now, for each $n \in \mathcal{N}$ with $f^{-1}(m_0) = n$, let \begin{equation}\label{n'lambda} n' = \frac{3 \cdot 5 \cdot 29^{5}}{\gcd(2 \cdot 5 \cdot 7 \cdot 11 \cdot 13 \cdot 17 \cdot 19 \cdot 23, m_0)} \ n \prod_{\substack{29  < p \leq 2^C \\ p \ \mathrm{prime} \\ p \nmid m_0}} p.\end{equation} Since $n$ is $2$-dense, it must be the case that $3 \mid n$, hence $3^2 \| n'.$ Also, we must have $5 \| n'$, since if $5 \mid n$ then $5 \mid m_0$, so it is removed in the denominator of $n'$. Thus, the only $5$ that appears in the factorization of $n'$ is the one that appears in the numerator of \eqref{n'lambda}. Now, $n'$ does not have any other prime factors smaller than $29$, since if $n$ is divisible by a prime $q < 29$, then $q \mid m_0$, hence $q \mid \gcd(2 \cdot 5 \cdot 7 \cdot 11 \cdot 13 \cdot 17 \cdot 19 \cdot 23, m_0).$ Thus, $n'$ is not $\varphi$-practical, since the absence of small primes (aside from the divisors of $45$) makes it so that $x^n-1$ has no divisor of degree $22$. However, we will show that $n'$ is $\lambda$-practical. Let $$l = n \prod_{\substack{29 < p \leq 2^C \\ p \ \mathrm{prime} \\ p \nmid m_0}} p.$$ Since $n$ is $2$-dense then $2 \mid n.$ Moreover, if we enumerate the prime factors of $l$ in increasing order, where $p_i$ is the $i^{th}$ smallest, then Bertrand's postulate implies that all of the primes $p_i$ dividing $l$ that are greater than $29$ satisfy $p_{i+1} \leq 2 p_{i}$. Thus, they satisfy the inequality given in Lemma \ref{necessary} as well. Let $$r = \frac{3 \cdot 5 \cdot 29^{5}}{\gcd(2 \cdot 5 \cdot 7 \cdot 11 \cdot 13 \cdot 17 \cdot 19 \cdot 23, m_0)},$$ so $n' = l \cdot r.$ Multiplying $l$ by $r$ does not prevent the primes greater than or equal to $29$ from satisfying the inequality from Lemma \ref{necessary}, since $29^5 > \frac{1}{3} \cdot 7 \cdot 11\cdot 13 \cdot 17 \cdot 19 \cdot 23,$ so that $r > 1.$ In other words, if a prime factor $p$ of $n'$ satisfies $p \leq m + 2$, then it is certainly the case that $p \leq mr + 2.$ Thus, we have just shown that $n'$ has the following structure: $n' = 45M$, where $P^-(M) = 29$ and all of the prime factors of $45M$ satisfy the inequality from Lemma \ref{necessary}. Since $45$ is $\lambda$-practical, then Lemma \ref{weak lambda} implies that $n'$ is $\lambda$-practical. Now, since we multiplied every $n$ in the pre-image of $m_0$ by the same number, there is a one-to-one correspondence between $\lambda$-practical numbers up to $r \cdot 4^{2^C} X$ that we have constructed and the $2$-dense numbers in the pre-image of $m_0$. As a result, at least $\frac{\kappa}{4^{2^C}} \frac{X}{\log X}$ of the integers up to $r \cdot 4^{2^C} X$ are $\lambda$-practical.\end{proof}

Note: By Lemma \ref{lnecessary}, all of the $\lambda$-practical numbers that we have constructed in the proof of Theorem \ref{philambdatheorem} are weakly $\varphi$-practical. As a result, this argument also shows that, for $X$ sufficiently large, there are $\gg \frac{X}{\log X}$ weakly $\varphi$-practicals in $[1,X]$ that are not $\varphi$-practical. 

\section{Proof of Theorem \ref{plambdatheorem}}

Our next endeavor will be to quantify the $p$-practical numbers that fail to be $\lambda$-practical. We begin with the following analogue of Lemma \ref{lnecessary}, which is proven in the same manner as its predecessor.

\begin{lemma}\label{pnecessary} Let $n = mq$, where $m$ is $p$-practical and $q$ is a prime satisfying $\ell_p^*(q) \leq m + 1$, with $(q, m) = 1.$ Then $n$ is $p$-practical. Moreover, if $n = m q^k$ where $k \geq 2$, then $n$ is $p$-practical if $\ell_p^*(q) \leq m.$ \end{lemma}

We can use Lemma \ref{pnecessary} to prove our first result comparing the sizes of the sets of $\lambda$-practical and $p$-practical numbers. 

\begin{proposition}\label{pinfinite} For each prime $p$, there are infinitely many $p$-practicals that are not $\lambda$-practical. \end{proposition}

\begin{proof} 

\textit{Case 1:} If $p = 2$, let $n = 21 \cdot \prod_{7 < q \leq X} q.$ Since $21$ is $2$-practical and since each prime $q_0$ in the product over $q$ satisfies the inequality $q_0 \leq \prod_{7 < q < q_0} q + 2$, then $n$ is $2$-practical by Lemma \ref{pnecessary}. However, $n$ is not $\lambda$-practical, since we cannot write $4$ in the form described in Theorem \ref{lambda}. Thus, by letting $X$ tend to infinity, we see that this method will generate an infinite family of $2$-practical numbers that are not $\lambda$-practical. 

\textit{Case 2:} For each prime $p \geq 3$,we will show that there exists a prime $q_0 \neq 3$ dividing $(p^2 + p + 1)$ and, for this $q_0$, the number $2q_0$ is $p$-practical but not $\lambda$-practical. First, observe that if such a prime exists, then $q_0 \mid (p^2 + p + 1)$ implies that $p^3 \equiv 1 \pmod{q_0}$, i.e. $\ell_p^*(q_0) \leq 3.$ Thus, since $m_0 = 2$ is $p$-practical for all primes $p$ and $\ell_p^*(q_0) \leq m_0 + 1$, then $2q_0$ is $p$-practical by Lemma \ref{pnecessary}. However, $2q_0$ is not $\lambda$-practical, since $q_0 > 3$ implies that $\lambda(q_0) \geq 4$. Thus, we cannot write $3$ in the form described in Theorem \ref{lambda}.

Now, we will prove the existence of a prime $q_0$ satisfying the conditions given above. The argument boils down to proving that $p^2 + p + 1$ is not a power of $3$. In the case where $p = 3$, we have $p^2 + p + 1 = 13$. Suppose that $p > 3$. Then, it must be the case that $p \equiv \pm 1 \pmod{3}.$ If $p \equiv -1 \pmod{3}$, then $p^2 + p + 1 \equiv 1 \pmod{3}$, so $p^2 + p + 1$ is not divisible by $3.$ On the other hand, if $p \equiv 1 \pmod{3}$ then, if $p^2 + p + 1$ were a power of $3$, the fact that $p > 3$ forces $p^2 + p + 1$ to be divisible by $9$. However, the congruence $x^2 + x + 1 \equiv 0 \pmod{9}$ has no solutions.\end{proof}

The infinite families that we have just constructed will play an important role in the proof of Theorem \ref{plambdatheorem}, which we will present at the end of this section. We remark that we could also have proven the second case in Proposition \ref{pinfinite} using the following lemma:

\begin{lemma}\label{binomial} If $n = p^k$ with $k \geq 0$, then $n$ is $p$-practical. \end{lemma}

\begin{proof} Let $n = p^k$, with $k \geq 0$. Using the binomial theorem, we have $x^{p^k} - 1 = (x-1)^{p^k}$ in $\F_p[x].$ Hence, $x^n-1$ has a divisor of every degree, so $n$ is $p$-practical.\end{proof}

In order to generate an infinite family of $p$-practical numbers when $p \geq 5$, we could simply have taken $n = p^k$, where $k$ ranges over all positive integers. By Lemma \ref{binomial}, $n$ is $p$-practical. However, when $p$ is in this range, we have $\lambda(p) = p - 1 \geq 4$. In other words, the gap between $1$ and $\lambda(p)$ is too large for $p^k$ to be $\lambda$-practical. In the case where $p = 3$, we can take $n = p^k$ with $k \geq 2.$ Then $4$ cannot be written in the form described in Theorem \ref{lambda}, so $n$ fails to be $\lambda$-practical. 

\begin{proof}[Proof of Theorem \ref{plambdatheorem}] This proof is nearly identical to the proof of Theorem \ref{philambdatheorem}. The main difference arises in our construction of $n'$, which varies depending on our choice of $p$. If $p = 2$, we let $$n' = \frac{7^2}{\gcd(10, m_0)} \ n \prod_{\substack{7 < q \leq 2^C \\ q \nmid m_0}} q,$$ where $C$, $m_0$ and $n$ are defined as in the proof of Theorem \ref{philambdatheorem}. Then $n'$ is of the form $n' = 21 \cdot m$, where $P^-(m) \geq 7$ and the primes dividing $m$ satisfy the weakly $\varphi$-practical conditions. As we showed in the proof of Lemma \ref{pinfinite}, $21$ is $2$-practical but not $\lambda$-practical. Thus, since all of the prime factors of $m$ are at least $7$, it follows from Lemma \ref{weak lambda} that $n'$ is not $\lambda$-practical. To show that $n'$ is $2$-practical, we use Lemma \ref{pnecessary} in place of Lemma \ref{lnecessary} in the proof of Theorem \ref{philambdatheorem}. 

The arguments for $p \geq 3$ follow the same line of reasoning. In the case where $p = 3$, we define $$ n' = \frac{13^4}{\gcd(3 \cdot 5 \cdot 7 \cdot 11 \cdot 13, m_0)} \ n \prod_{\substack{13 < q \leq 2^C \\ q \nmid m_0}} q.$$ Then $n'$ is of the form $n' = 26 \cdot 13^3 \cdot m$, where $P^-(m) > 13$ and all of its prime factors satisfy the weakly $\varphi$-practical conditions. Since $26$ is $3$-practical but not $\lambda$-practical, we can use the same reasoning from the $p = 2$ case to show that $n'$ is indeed $3$-practical but not $\lambda$-practical. In the case where $p \geq 5$, we let \begin{equation}\label{n'5practical} n' = \frac{2p^2}{\gcd(6p, m_0)} \ n \prod_{\substack{p < q \leq 2^C \\ q \nmid m_0}} q.\end{equation} Since $m_0$ is $2$-dense, it must be the case that $m_0$ is divisible by $6$ and by at least one of the primes $5$ and $7$. In other words, there are three cases that we need to consider: If $42 \mid m_0$, then \eqref{n'5practical} yields $n' = 14p^2m$. If $30 \mid m_0$ but $7 \nmid m_0$, we have $n' = 10p^2m$. If $210 \mid m_0$, then $n' = 70p^2m$. Either way, $n'$ is not $\lambda$-practical, since $10, 14$ and $70$ are not $\lambda$-practical and $P^{-}(m) > p\geq 5.$ As in the cases above, we can use the proof of Theorem \ref{philambdatheorem} with Lemma \ref{pnecessary} to show that $n'$ is $p$-practical.\end{proof}

\textit{Acknowledgements.} This work was completed as part of my Ph.D. thesis \cite{thompsonthesis}. I would like to thank my adviser, Carl Pomerance, for suggesting many of the problems mentioned in this paper and for providing endless encouragement and help with editing. I would also like to thank Greg Martin for posing the question that is discussed in Proposition \ref{union}. 


\end{document}